\documentclass{amsart}

\usepackage{amssymb,amsmath,mathrsfs,amsthm,amscd,esint}
\usepackage[dvipsnames]{xcolor}

\newcommand{\R}{\mathbb{R}}

\newcommand{\N}{\mathbb{N}}

\newcommand{\cA}{{\mathscr{A}}}
\newcommand{\cB}{{\mathcal{B}}}

\newcommand{\cF}{{\mathcal{F}}}
\newcommand{\cG}{{\mathcal{G}}}

\newcommand{\cM}{{\mathscr{M}}}

\newcommand{\Loneloc}{{L^1_{\text{loc}}}}

\DeclareMathOperator*{\esssup}{ess\,sup}
\DeclareMathOperator*{\essinf}{ess\,inf}

\def\BMO{\mathrm{BMO}}

\newtheorem{theorem}{Theorem}[section]
\newtheorem{lemma}[theorem]{Lemma}

\theoremstyle{definition}
\newtheorem{definition}[theorem]{Definition}

\theoremstyle{remark}
\newtheorem{remark}[theorem]{Remark}

\numberwithin{equation}{section}

\begin{document}

\title [Decreasing rearrangement and mean oscillation]
{A note on decreasing rearrangement and mean 
oscillation on measure spaces}

\author[Burchard]{Almut Burchard}
\address{(A.B.) University of Toronto, Department of Mathematics, Toronto, ON M5S 2E4, Canada}
\curraddr{}
\email{almut@math.toronto.edu}

\author[Dafni]{Galia Dafni}
\address{(G.D.) Concordia University, Department of Mathematics and Statistics, Montr\'{e}al, QC H3G 1M8, Canada}
\curraddr{}
\email{galia.dafni@concordia.ca}
\thanks{A.B. was partially supported by Natural Sciences and Engineering Research Council (NSERC) of Canada. G.D. was partially supported by the Natural Sciences and Engineering Research Council (NSERC) of Canada and the Centre de recherches math\'{e}matiques (CRM). 
R.G. was partially supported by the Centre de recherches math\'{e}matiques (CRM), the Institut des sciences math\'{e}matiques (ISM), and the Fonds de recherche du Qu\'{e}bec -- Nature et technologies (FRQNT)}

\author[Gibara]{Ryan Gibara}
\address{(R.G.) Universit\'{e} Laval, D\'{e}partement de math\'{e}matiques et de statistique, Qu\'{e}bec, QC G1V 0A6, Canada}
\curraddr{}
\email{ryan.gibara@gmail.com}


\date{}

\begin{abstract} We derive bounds on
the mean oscillation of
the decreasing rearrangement $f^*$ on $\R_+$
in terms of the mean oscillation
of $f$ on a suitable 
measure space $X$. In the special case of
a doubling metric measure space, the bound depends only on 
the doubling constant.
\end{abstract}

\maketitle


\section{Introduction}


The decreasing rearrangement of a 
real-valued measurable function $f$
is the unique monotone decreasing 
function $f^*$ on the positive half-line
that is right-continuous and equimeasurable 
with $|f|$. In essence, $f^*$ is a model of $f$
where all geometric information about the level sets has been stripped away. The study of this rearrangement goes back to the work of Hardy-Littlewood \cite{hl} and Hardy-Littlewood-P\'{o}lya \cite{hlp}. 

The decreasing rearrangement is a nonlinear operator that is both isometric and non-expansive on $L^p$-spaces. Moreover, the norms in many other commonly-used function spaces, including the Lorentz spaces $L^{p,q}$ and the Orlicz spaces $L^\phi$, are invariant under equimeasurable rearrangements, making the decreasing rearrangement a valuable tool \cite{bs,cia,fv}. 

In this paper, we consider inequalities
that bound the mean oscillation of $f^*$ 
in terms of the mean oscillation of $f$.
Spaces of functions of bounded mean oscillation (BMO)
are useful as replacements for $L^\infty$
in estimates for singular integral operators and Sobolev 
embedding theorems. The BMO condition on locally integrable functions on $\R^n$,
as introduced in 1961 by John and Nirenberg \cite{jn},
is a uniform bound on their 
mean oscillation over cubes.

Our setting for the present paper is a measure space $X$. We define $\BMO$
as the set consisting
of those functions satisfying a uniform bound on their mean 
oscillation over a specified collection 
$\cA$ of measurable sets. We provide sufficient
conditions for the decreasing rearrangement $f^*$ of a 
function $f\in \BMO(X)$ to have bounded mean oscillation
on the corresponding interval $(0,\mu(X))$ when $X$ is semi-finite. 

We then restrict the setting to metric measure spaces, where
it is natural to study $\BMO$ over the collection of all balls.
For some references on $\BMO$ on a metric measure space, 
see \cite{kkms,mmno} and \cite[Chapter 3]{bb}. 
Our general result gives us the following theorem for metric measure spaces.
\begin{theorem} [Boundedness in doubling spaces]
\label{thm:bounded-doubling}
Let $X$ be a
doubling metric measure space. There exists
a constant $c_*\ge 1$, depending only on the doubling constant, 
such that if $f^*$ is the decreasing rearrangement
of a function $f\in\BMO{}(X)$, then
it is locally integrable and satisfies
\begin{equation}
\label{eq:BMO-bounded}
\|f^*\|_{\BMO} \leq c_* \|f\|_{\BMO}\,.
\end{equation}
\end{theorem}

In the case where $\mu(X)=\infty$, the boundedness 
\eqref{eq:BMO-bounded} can be derived from 
a result of Aalto~\cite{aa} which says that
functions of bounded mean oscillation
on a doubling measure metric space are
in weak-$L^\infty$. The space weak-$L^\infty$,
originally introduced in \cite{bdvs}, consists 
of all functions having finite decreasing rearrangement 
such that $f^{**}-f^*\in L^\infty(0,\mu(X))$.
Here, $f^{**}$ is the maximal function of 
$f^*$ as defined in \cite{hl} -- see equation (8.2) 
therein and the remark that follows it.

Our techniques do not go through weak-$L^\infty$ and 
apply regardless of whether $\mu(X)$ is finite or infinite. The main tool for the proof is a bound on the oscillation of
$f^*$ on intervals in terms of the oscillation of $f$ on a 
collection of balls obtained from a covering argument
of Calder\'{o}n-Zygmund type. This estimate is 
inspired by the work of Klemes \cite{kl}.
We calculate explicitly the 
dependence of the constant $c_*$ on the doubling constant 
of the measure $\mu$:
\begin{equation}
\label{eq:cstar}
c_*=\inf_{\lambda>3}\sup_{x\in X, r>0} \frac{\mu(B(x,\lambda r))}{\mu(B(x,r))}.
\end{equation}

The basic estimate, Lemma~\ref{lem:osc} below, applies to bounded functions, whose decreasing rearrangement is 
automatically locally integrable. In order to extend this 
to all functions of bounded mean oscillation, we approximate 
by either bounded functions or integrable functions via truncation. 
Note that the term `approximate' is 
misleading insofar as neither bounded 
functions nor integrable functions are dense in $\BMO$, and these 
truncations do not converge in the $\BMO$-seminorm. 
Nevertheless, the pointwise convergence is monotone 
and so estimates can be inferred from the truncation via 
the monotone convergence theorem.

In the first part of the paper, we establish a collection of lemmas which allow us to deduce mean oscillation bounds on the decreasing rearrangement of general rearrangeable functions in $\BMO(X)$ from corresponding estimates for nonnegative bounded functions, without going through the absolute value. The sharp bound $\|f^*\|_{\BMO(0,1)}\leq  \|f\|_{\BMO(0,1)}$,
which follows from work of Klemes \cite{kl} and Korenovskii \cite{ko1}, plays an important role. In upcoming related work~\cite{BDG}, we use these techniques to improve the constant in the bound 
\begin{equation}
\label{eq:Cn}
\|f^*\|_{\BMO(\R_+)}\leq C_n \|f\|_{\BMO(\R^n)}\,, \qquad n\ge 1\,,
\end{equation}
to $C_n=2^{\frac{n+1}{2}}$ (previous results of Bennett--DeVore--Sharpley~\cite{bdvs} imply Eq.~\eqref{eq:Cn} with $C_n=2^{n+5}$). 


\section{Preliminaries and notation}


\subsection{Decreasing rearrangement}

Let $(X,\cM, \mu)$ be a nontrivial measure space.  Given a real-valued measurable function $f$ on $X$,
define its {\em distribution function} as
$$
\mu_f(\alpha)=\mu(E_\alpha(f))\,,\qquad \alpha\ge 0\,,
$$
where $E_\alpha(f)=\{x\in X:|f(x)|>\alpha \}$
is the level set of $|f|$ at height $\alpha$. The function
$\mu_f:[0,\infty)\rightarrow[0,\infty]$ is decreasing and right-continuous. 
Note that throughout this paper, 
`decreasing' should be understood in the sense of `nonincreasing'.

Let $f$ be a real-valued measurable function
satisfying $\mu_f(\alpha)\rightarrow{0}$ as 
$\alpha\rightarrow\infty$. We call such a function {\it rearrangeable}, and define its {\em decreasing rearrangement} by
$$
f^\ast(s)= \inf\{\alpha\geq{0}:\mu_f(\alpha)\leq{s} \}\,, \qquad s> 0\,.  
$$
The value of $f^*(s)$ is finite for every $s>0$
because the set $\{\alpha\geq{0}:\mu_f(\alpha)\leq{s}\}$ 
is nonempty by assumption. As is the case with the distribution function, 
$f^*$ is decreasing and right-continuous. 
Note that if $\mu(X)<\infty$, then every real-valued measurable function 
is rearrangeable and we consider $f^*$ as a function on $(0,\mu(X))$. 

By construction, $f^*$ is {\it equimeasurable} with $|f|$: 
$$
|\{s\in(0,\mu(X)):f^\ast(s)>\alpha\}| = \mu_{f}(\alpha)\,,\qquad\alpha \ge 0\,,
$$
where we use $|\cdot|$ to denote Lebesgue measure. As a consequence, $f^\ast\in L^p(0,\mu(X))$ if and only if $f\in L^p(X)$, with $\|f^*\|_p=\|f\|_p$ for all $p\in [1,\infty]$. In particular, if $f\in L^\infty(X)$, then $f^\ast\in L^\infty(0,\mu(X))\subset \Loneloc(0,\mu(X))$. Note also that any $f\in L^p(X)$ for $p\in [1,\infty]$ is rearrangeable.   

The Hardy-Littlewood inequality \cite[Theorem 378]{hlp} states that $\int f^*g^*\ge \int |fg|$ for a pair of integrable functions $f$ and $g$.
This inequality is fundamental and extends to other classical rearrangements,
including Steiner symmetrization, two-point symmetrization,
and the symmetric decreasing rearrangement (see, for instance, \cite{bh}). A special case is the following.

\begin{lemma}[Hardy-Littlewood inequality]
If $f$ is integrable on $A\in\cM$, then
$$
\int_0^{\mu(A)}\!f^* \ge \int_A \! |f|\,.
$$
\end{lemma}

On several occasions, when $\mu(X) < \infty$, we will use another one-dimensional rearrangement. The {\em signed decreasing 
rearrangement} of a real-valued measurable function $f$ is defined by
$$
f^\circ(s)= (f_+)^*(s) - (f_-)^*(\mu(X)-s)\,, \qquad 0<s<\mu(X)\,.  
$$
Note that
\begin{equation}
\label{eq:circ-pm}
(-f)^\circ(s)= -f^\circ(\mu(X)-s)\,
\end{equation} 
holds for all but countably many $s \in (0,\mu(X))$. Moreover,
$f^\circ$ is equimeasurable with $f$:
$$
\bigl|\{s\in (0,\mu(X)):f^\circ(s)>\alpha \}\bigr|
=\mu(\{x\in X:f(x)>\alpha \})\,,\qquad \alpha\in\R \,.
$$
In general, $(f^\circ)^\ast=f^\ast$. 
If $f\geq{0}$ almost everywhere, then $f^\circ=f^\ast$.  

For functions which are bounded from either above or below, we can express $f^\circ$ in terms of the decreasing rearrangement, as shown by the following lemma.
This also allows us to rewrite $f^*$ in terms of the rearrangement of a vertical shift of $f$, avoiding $|f|$.

\begin{lemma}[Signed decreasing rearrangement]
\label{lem:up}
Let $f$ be a real-valued measurable function on 
a finite measure space $X$.
If $\essinf f \ge -\beta$ for some $\beta> 0$, then
$f^\circ(s)=(f+\beta)^*(s)-\beta
$
for all but countably many $s\in (0,\mu(X))$. In particular,
\begin{equation}
\label{eq:beta}
 f^* = ((f+\beta)^*-\beta)^*\,.
 \end{equation}
If, instead, $\esssup f \le \beta$, then
\begin{equation}
\label{eq:flipped}
f^\circ(s)=\beta-(\beta-f)^*(\mu(X)-s)
\end{equation}
for all but countably many $s\in (0,\mu(X))$, and $f^* = (\beta-(\beta-f)^*)^*$.
\end{lemma}

\begin{proof} By definition of the decreasing
rearrangement, the function $g:= (f+\beta)^*-\beta$ is decreasing,
right-continuous, and takes values in $[-\beta,\infty)$.
Its level sets satisfy
$$
\bigl|\{s\in (0,\mu(X)): g(s) >\alpha\}\bigr| = \mu(\{x\in X: f(x) >\alpha\})\,,
$$
since both agree with $\mu_{f+\beta}(\alpha+\beta)$.
This determines
$g$ uniquely among decreasing, right-continuous
functions equimeasurable with $f$. 
Thus, $g$ agrees with $f^\circ$, except possibly at jump discontinuities,
of which there are at most countably many.

The second claim follows by replacing $f$ with $-f$ and 
using Eq.~\eqref{eq:circ-pm}.  Note that $f^*$ is also the decreasing rearrangement of the reflected function $f^\circ(\mu(X)-s)$.
\end{proof}

Finite measure is essential to the identity Eq.~\eqref{eq:beta}.  If $X$ has infinite measure and $\|f\|_{L^\infty}\le \beta$,
then $f^*\ge ((f+\beta)^*-\beta)^*$,
and the inequality is typically strict.
For instance, if $f(x)=-(\sin x)\mathcal{X}_{(0,\pi)}(x)$ on $\R_+$, then
$f^*(s)=\bigl(\cos\frac{s}{2}\bigr) \mathcal{X}_{(0,\pi)}(s)$,
while $((f+1)^*-1)^*=0$.

For more details on the decreasing rearrangement, we refer to \cite{sw}.

\subsection{Bounded mean oscillation}

Let $(X,\cM,\mu)$ be a measure space and let $f$ be a 
real-valued measurable function on $X$ that is 
integrable on $A$, where $A \in \cM$ with
$0<\mu(A)<\infty$. Denote by $f_A:= \fint_A f$ the {\em mean} of $f$ on $A$,
and by 
$$
\Omega(f,A):=\fint_{A}\!|f-f_A|
$$
the {\em mean oscillation} of $f$ on $A$. By the definition of the mean, the mean
oscillation can also be computed as
\begin{equation}
\label{eq:osc-alt}
\Omega(f,A)= 2\fint_{A}(f-f_A)_+ =2\fint_{A}(f-f_A)_-\,.
\end{equation}

A {\em basis} in $X$ is a collection $\cA \subset \cM$, with 
$0<\mu(A)<\infty$ for every $A\in \cA$, whose union covers~$X$. 
\begin{definition}\label{def:BMO}
Let $\cA$ be a basis in $X$.
We say that a function satisfying $f\in L^1(A)$ for all $A\in \cA$
is of {\em bounded mean oscillation} if 
\begin{equation}\label{eq:bmo}
\|f\|_{\BMO}:=\sup_{A\in \cA}\Omega(f,A)<\infty\,.
\end{equation}
The vector space of functions of bounded mean oscillation
is denoted by $\BMO(X)$.
\end{definition}

This abstract $\BMO(X)$ space has no \emph{a priori} geometry as the sets in the basis $\cA$ can be quite pathological.  
Since $\Omega(f+\alpha,A)=\Omega(f,A)$ for any $\alpha\in\R$,
Eq.~\eqref{eq:bmo} defines
only a seminorm on $X$ that vanishes 
on constant functions. Additional assumptions on
$X$ and $\cA$ are required to ensure that 
the quotient of $\BMO$ modulo constants is a Banach space (see, for instance, the considerations found in \cite{dgl}).

On a metric measure space, the natural choice for the basis $\cA$ is the collection of all balls, $\cB$ (see Section~\ref{sec:mms}). On $\R^n$ with Lebesgue measure, the resulting $\BMO$ space is equivalent, with seminorms that
agree up to a dimension-dependent factor, to the
classical one, defined with the basis of all cubes with
sides parallel to the axes. 

For a general reference on $\BMO$ functions on Euclidean space, see \cite{ko2}.


\subsection{Decreasing rearrangement 
and BMO in one dimension}

We collect here some properties of $\BMO$ functions 
in one dimension that will prove to be useful below. We 
maintain the convention that when $X\subset\R$ is an interval, 
the measure is the Lebesgue measure on $X$ and the basis is the 
collection of all non-empty finite open subintervals of $X$. 

\begin{lemma}\label{normcalculation}
If $f$ is a decreasing locally integrable
function on an open interval $X\subset \R$, then
$$
\|f\|_{\BMO}=\begin{cases} 
\max\left\{ \sup\limits_{0<t< T} \Omega(f,(0,t)), \sup\limits_{0< t<T} \Omega(f,(t,T))\right\}\,, 
\quad & X=(0,T)\,,\\
\sup\limits_{t>0}\Omega(f,(0,t))\,, & X=\R_+\,,\\
\frac12 (\sup f -\inf f)\,, & X=\R\,.
\end{cases}
$$
\end{lemma}

\begin{proof}
The proofs for the cases $X=\R_+$ and $X=\R$ can be 
found in \cite[Lemma 2.22 and Proposition 2.26, respectively]{ko2}.

For the case $X=(0,T)$, we will show
that for any interval $I=(a,b)\subset{X}$, there exists 
another interval $J$ of either the form $(0,t)$ with
$0<t\leq T$, or $(t,T)$ with $0\leq t<T$, such that 
$\Omega(f,I)\leq \Omega(f,J)$. By the monotonicity of $f$, 
it suffices to find such an interval $J\supset I$ 
for which $f_{J}=f_{I}$ (see \cite{kl} or \cite[Property 2.15]{ko2}). 

Since $f$ is decreasing, $f_{(a,T)} \leq f_I \leq f_{(0,b)}$.
If $f_{(0,T)}=f_I$, then we take $J=(0,T)$. 
If $f_{(0,T)}<f_I\leq f_{(0,b)}$, then by continuity
in $b$ the intermediate value theorem
yields an interval of the form $J=(0,t)$
with $f_J=f_I$. If 
$f_{(0,T)}>f_I\geq f_{(a,T)}$, then continuity 
in $a$ yields an interval of the form $J=(t,T)$.
We finally appeal to continuity once
more to restrict the suprema to proper subintervals with $0<t<T$. 
\end{proof}

When $X\subset\R$ is a finite interval, the following sharp result is known. 

\begin{theorem}[Klemes-Korenovskii Theorem]
If $f\in \BMO{(a,b)}$, then $f^\ast\in\BMO(0,b-a)$ with $\|f^*\|_{\BMO} \le \|f\|_{\BMO}$. 
\end{theorem}

The proof is an immediate consequence of the combination of results of Klemes \cite{kl} and Korenovskii \cite{ko1}. Klemes shows that if $f\in \BMO{(a,b)}$ then $f^\circ\in\BMO(0,b-a)$ with $\|f^\circ\|_{\BMO} \le \|f\|_{\BMO}$. This is clearly sharp, in the sense that the constant 1 cannot be decreased. Korenovskii uses this result to show that $|f|\in\BMO(a,b)$ if $f\in \BMO{(a,b)}$ with $\||f|\|_{\BMO} \le \|f\|_{\BMO}$. Again, this is clearly sharp, and is an improvement on the trivial bound $\||f|\|_{\BMO} \le 2\|f\|_{\BMO}$.  

\section{Technical tools}

\subsection{Truncation}

Truncation is an important tool for this paper as it often allows for the reduction of a proof to the case of bounded functions. The following lemma demonstrates that truncation is among a class of transformations that behave well with respect to both rearrangement and mean oscillation.

\begin{lemma}\label{trunc}
Let $\phi:\R\rightarrow\R$ be increasing. 
\begin{enumerate}
\item If $\phi$ is odd and $f$ is rearrangeable, then $(\phi\circ f)^*=\phi\circ f^*$.
\item If $\phi$ is non-expansive and $f$ is integrable on $A \in \cM$ with $0<\mu(A)<\infty$, then
$$
\Omega(\phi\circ f,A)\le \Omega(f,A)\,.
$$
\end{enumerate}
\end{lemma}

\begin{proof} 
Since $\phi$ is increasing, it follows from 
\cite[Property (v) in Section 3.3]{ll} 
that $\phi\circ f^*=(\phi\circ |f|)^*$. (The property
is stated for the symmetric decreasing 
rearrangement, but holds, with the same proof,
for the decreasing rearrangement.) 
Therefore,
$$
\phi\circ f^* =(\phi\circ |f|)^* = |\phi\circ f|^*=(\phi\circ f)^*\,,
$$
showing (1). We have used that $\phi$ is odd in the second step.

If $(\phi\circ f)_A \ge \phi(f_A)$, we use that $\phi$ is 
non-expansive to obtain
$$
\fint_A (\phi\circ f -(\phi\circ f)_A)_+
\le \fint_A \bigl(\phi\circ f -\phi(f_A)\bigr)_+ 
\le \fint_A (f-f_A)_+
$$
as $f(x)> f_A$ if $\phi\circ f(x)> \phi(f_A)$. The first formula 
in Eq.~\eqref{eq:osc-alt} yields the result. 
If $(\phi\circ f)_A < \phi(f_A)$, 
we similarly estimate the second formula in Eq.~\eqref{eq:osc-alt} to obtain (2). 
\end{proof}

When $\phi$ is non-expansive but not monotone,
then by writing it as the difference of two
monotone functions, one has that
$\Omega(\phi\circ f, A)\le 2\,\Omega(f, A)$. This is true, in particular, when $\phi(\alpha)=|\alpha|$.

For $\beta>0$, the functions
\begin{align}
\label{innertrunc}
\phi_\beta(\alpha)&:=\min\{\alpha_+,\beta\}-\min\{\alpha_-,\beta\}\\
\label{outertrunc}
\phi^\beta(\alpha)&:=(\alpha-\beta)_+-(\alpha-\beta)_-
\end{align}
satisfy the hypotheses of both (1) and (2) in the preceding lemma. The truncations $\phi_{\beta}\circ f$ and $\phi^\beta\circ f$ commute with rearrangement and reduce mean oscillation. The decomposition $f= \phi_{\beta}\circ f+\phi^{\beta}\circ f$ splits $f$ horizontally into a bounded function $\phi_{\beta}\circ f$ taking values between $-\beta$ and $\beta$, and the remainder $\phi^{\beta}\circ f$.  

As mentioned above, bounded functions are not dense in $\BMO$, and truncations do not approximate 
$f\in\BMO{}(X)$ in the seminorm.
Nevertheless, these truncation techniques, in the form of the following lemma, allow us to transfer bounds for the rearrangements of nonnegative functions to bounds for sign-changing functions, without increasing the constants

\begin{lemma}\label{lem:pm} Let $X$ be a semi-finite measure space with
$\mu(X)=\infty$. Then,
$$\|f^*\|_{\BMO} \leq \max \left\{ \|(f_+)^*\|_{\BMO}, \|(f_-)^*\|_{\BMO}
\right\}$$
for any function $f\in L^\infty(X)$.
\end{lemma}

\begin{remark}
By the proof of Case 3 of Lemma~\ref{lem:BMO-bound} below, the conclusion of Lemma~\ref{lem:pm} is valid for all rearrangeable $f\in \BMO{}(X)$.
\end{remark}


\begin{proof}
By Lemma~\ref{normcalculation}, it suffices to show that
\begin{equation}
\label{eq:pm}
\Omega(f^*,(0,t))\le 
\max \left\{ \|(f_+)^*\|_{\BMO}, \|(f_-)^*\|_{\BMO}\right\}
\end{equation}
for all $t>0$.

If $f^*$ is constant on $(0,t)$, there is nothing
to show. Otherwise, $f^*(0)>f^*(t)$.
Consider $\phi^\beta(\alpha)$ as in Eq. \eqref{outertrunc}, with $\beta:=f^*(t)$. Since $f^*(s)\geq\beta$ on $(0,t)$, part (1) of Lemma \ref{trunc} gives us that $\Omega((\phi^{\beta}\circ f)^*, (0,t))=\Omega(f^*-\beta,(0,t))=\Omega(f^*,(0,t))$. Moreover, as $(\phi^\beta\circ f)_{\pm}=\phi^\beta\circ(f_{\pm})$, both parts of Lemma~\ref{trunc} imply that 
$$
\|((\phi^\beta \circ f)_{\pm})^*\|_{\BMO} =\|\phi^\beta \circ (f_{\pm})^*\|_{\BMO}\leq \|(f_{\pm})^*\|_{\BMO}\,,
$$
showing that composition with $\phi^\beta$ reduces the right-hand side of Eq.~\eqref{eq:pm}. Thus, we may assume, without loss of generality, that $f^*(t)=0$. That is, $\mu_f(0)=|E_0(f)|\le  t$ and $f^\ast\in L^1(\R_+)$. 

The assumption that $X$ is semi-finite but not finite ensures the existence of a measurable subset $A\subset X$ containing $E_0(f)$ with $t<\mu(A)<\infty$ - see \cite[Exercise 1.14]{fol}. Since $f^*$ vanishes outside $(0,t)$, its restriction to
$(0,\mu(A))$ satisfies
$$ 
f^*\big\vert_{(0,|A|)}=(f\big\vert_A)^*\,.
$$
Writing $g:=f\big\vert_A$, consider its signed decreasing rearrangement,
$$
g^\circ(s) = (f_+)^*(s)-(f_-)^*(\mu(A)-s)\,,\qquad s\in (0,\mu(A))\,.
$$
As $(g^\circ)^*=(f\big\vert_A)^*=f^*\big\vert_{(0,\mu(A))}$, the Klemes-Korenovskii theorem implies that
\begin{equation}
\label{eq:pm-Klemes}
\Omega(f^*, (0,t)) \le \|g^\circ\|_{\BMO(0,\mu(A))}\,.
\end{equation}

Next, we estimate the norm of $g^\circ$.
If $\mu(A)>2t$, then $g^\circ(s) = (f_+)^*(s)$ for all $0<s<t$ and $g^\circ(s) = -(f_-)^*(\mu(A)-s)$ for all $t<s<\mu(A)$. It follows that for $\tau \leq t$, we have
$$
\Omega(g^\circ, (0,\tau)) = \Omega((f_+)^*, (0,\tau))  \le  \|(f_+)^*\|_{\BMO}\,,
$$
and for $\tau > t$, we have 
$$
\Omega(g^\circ, (\tau,\mu(A))) = \Omega((f_-)^*, (0,\tau))  \le  \|(f_-)^*\|_{\BMO}\,.
$$

It remains to consider intervals of the form $(0,\tau)$ and $(\tau,\mu(A))$ 
that intersect both $\{s\in(0,\mu(A)): (f_+)^*(s)>0\}$ 
and $\{s\in(0,\mu(A)): (f_-)^*(s)>0\}$. As the measure of 
each of these sets is finite and independent of $\mu(A)$, we can 
take $\mu(A)$ large enough so that the mean oscillation of 
$g^\circ$ on either $(0,\tau)$ or $(\tau,\mu(A))$ is as small 
as we would like.  Therefore, it follows from 
Lemma~\ref{normcalculation} that
$$\|g^\circ\|_{\BMO(0,\mu(A))}\le  \max \left\{ \|(f_+)^*\|_{\BMO}, 
\|(f_-)^*\|_{\BMO}\right\}\,.$$
Inserting this into Eq.~\eqref{eq:pm-Klemes}
completes the proof of Eq.~\eqref{eq:pm}.
\end{proof}

\subsection{Limiting arguments}

The following 
result allows us to transfer mean oscillation properties of the decreasing rearrangement from an approximating sequence to its limit. It relies on the fact that both mean oscillation and the decreasing rearrangement behave well under monotone convergence.

\begin{lemma}
\label{lem:mono}
Let $f$ be  a rearrangeable function on $X$,
and suppose $\{f_k\}$ is a sequence of real-valued measurable
functions with $|f_k|\uparrow |f|$ pointwise and $f_k^* \in \Loneloc(0,\mu(X))$ for each $k$.
If for some $0<t<\mu(X)$,
$$
M:= \sup_{k} \Omega (f_k^*,(0,t))<\infty\,,
$$
then $f^*\in\Loneloc(0,\mu(X))$. Moreover, for every finite subinterval $I\subset(0,\mu(X))$,
$$
\Omega(f^*,I)\le \liminf_{k\to\infty} \Omega(f_k^*,I),
$$
and, in  particular, $\Omega(f^*, (0,t))\le M$.
\end{lemma}

\begin{proof} Since $E_\alpha(f_k)\subset E_\alpha(f_{k+1})\subset E_\alpha(f)$ and $\bigcup_{k} E_\alpha(f_k)=E_\alpha(f)$ for all $\alpha>0$, we have that $\mu_{f_k}(\alpha)\uparrow
\mu_f(\alpha)$ by continuity of measure from below
and that $f_k^*\uparrow f^*$ pointwise. 

We first show that $f^*$ is integrable over $(0,t)$.
Since $f_k^*$ is decreasing, its value at $\frac t 2 $
is a median for $f_k^*$ on $(0,t)$; i.e., 
it defines a constant function of minimal distance to $f_k^*$ in the $L^1$ norm. Therefore, 
$$
(f_k^*)_{(0,t)}-f_k^*\bigl(\tfrac{t}{2}\bigr)
\le \fint_0^t \!|f_k^*(s)-f_k^*( \tfrac t 2)|\, ds
\le \Omega(f_k^*, (0,t)) \le M\,.
$$
By monotone convergence, 
$$
f^*_{(0,t)} = \lim_{k\to\infty} (f_k^*)_{(0,t)}
\le \lim_{k\to\infty} 
f_k\bigl(\tfrac{t}{2}\bigr) +  M = f^*(\tfrac t 2) + M <\infty\,.
$$
Since $f^*$ is decreasing, it follows that
$f^\ast\in \Loneloc(0,\mu(X))$. 

On every finite interval $I\subset (0,\mu(X))$, 
the means satisfy
$(f_k^*)_I\uparrow f^*_I$ by monotone convergence.
Hence $|f_k^* - (f_k^*)_I|$ converges
pointwise to $|f^*-(f^*)_I|$, almost everywhere on $I$.
It follows from Fatou's lemma 
that $\Omega(f^*,I) \le \liminf \Omega(f_k^*, I)$.
\end{proof}

We now show how to extend mean oscillation bounds from nonnegative bounded functions to general rearrangeable functions in $\BMO$.

\begin{lemma}
\label{lem:BMO-bound}
Let $X$ be a semi-finite measure space and $c_*\ge 1$.
Assume that
\begin{equation}
\label{eq:bounded-proof-5}
\Omega(g^*,(0,t))\le c_* \|g\|_{\BMO}
\end{equation}
holds for every nonnegative $g\in L^\infty(X)$ and all $0 < t < \mu(X)$.

Then the decreasing rearrangement of every rearrangeable function $f\in\BMO(X)$
is locally integrable and 
\begin{equation}
\label{eq:BMO-bound}
\|f^\ast\|_{\BMO}\le c_*\|f\|_{\BMO}.
\end{equation}
\end{lemma}

\begin{proof}
Let $f$ be a rearrangeable function  in $\BMO(X)$.  

{\em Case 1: $\mu(X)=\infty$, $f\in L^\infty(X)$.}\ 
If $f$ is nonnegative, then Eq.~\eqref{eq:BMO-bound} follows from the hypothesis  Eq.~\eqref{eq:bounded-proof-5}
via Lemma~\ref{normcalculation}.

If $f$ changes sign, then $f_+$ and $f_-$ are nonnegative and bounded, and 
therefore 
$\|(f_{\pm})^*\|_{\BMO} \le c_* \|f_{\pm}\|_{\BMO}$. 
From Lemma~\ref{trunc}, $\|f_{\pm}\|_{\BMO}\leq \|f\|_{\BMO}$. 
An application of Lemma~\ref{lem:pm} gives
$$
\|f^*\|_{\BMO} \le 
\max\left\{\|(f_+)^*\|_{\BMO}, \|(f_-)^*\|_{\BMO}\right\}
\le c_*\|f\|_{\BMO} \,.
$$

\smallskip {\em Case 2: $\mu(X)<\infty$, $f\in L^\infty(X)$.}\ 
Set $\beta=\|f\|_{L^\infty}$.
If $f$ is nonnegative, we have, as in Case 1, that Eq.~\eqref{eq:bounded-proof-5}  holds for  $0<t<\mu(X)$.
Since $\mu(X)<\infty$, 
according to Lemma~\ref{normcalculation}
we must also bound the oscillation of $f^*$ on $(t,\mu(X))$ for $0<t<\mu(X)$.  
As $f^\ast=f^\circ$, Eq.~\eqref{eq:flipped} implies that
$$
f^*(s)=\beta-(\beta-f)^*(\mu(X)-s)\,,\qquad a.e.\ s\in (0,\mu(X))\,.
$$
From this and Eq.~\eqref{eq:bounded-proof-5} applied to $\beta-f$, which is nonnegative, we have that 
\begin{equation}
\label{eq:bounded-proof-2}
\Omega(f^*, (t,\mu(X))) =\Omega((\beta-f)^*, (0,\mu(X)-t))\le
c_* \|\beta-f\|_{\BMO} = c_* \|f\|_{\BMO} \,
\end{equation}
for all $0<t<\mu(X)$. 
By Lemma~\ref{normcalculation}, 
Eqs.~\eqref{eq:bounded-proof-5} and~\eqref{eq:bounded-proof-2}
yield Eq.~\eqref{eq:BMO-bound} for nonnegative bounded $f$.

If $f$ changes sign, we apply the arguments above to 
get Eq.~\eqref{eq:BMO-bound} for the nonnegative function 
$f + \beta$, which has the same $\BMO$ norm as $f$. This observation, 
Eq.~\eqref{eq:beta}, and the Klemes-Korenovskii theorem give
$$\|f^*\|_{\BMO} \le \|(f+\beta)^*-\beta\|_{\BMO}
= \|(f+\beta)^*\|_{\BMO}\le  c_*\|f\|_{\BMO}\,.
$$

\smallskip{\em Case 3: Rearrangeable $f\in\BMO{}{}(X)$.}\ 
We approximate $f$ by the truncations $f_k:=\phi_k\circ f$, where $\phi_k$ is defined as in Eq. \eqref{innertrunc}. 
By Lemma~\ref{trunc}, it follows that $\|f_k\|_{\BMO}\le \|f\|_{\BMO}$.
Since the $f_k$ are bounded, the previous cases apply and the decreasing 
rearrangements, denoted unambiguously by $f_k^\ast$,  satisfy
Eq.~\eqref{eq:BMO-bound}. 
Thus
$\Omega(f_k^\ast, I)\le  c_*\|f\|_{\BMO}$
for every finite interval $I\subset (0,\mu(X))$, in particular for every interval of the form
$(0,t)$ for some real number $t < \mu(X)$.  Applying
Lemma~\ref{lem:mono} for such a $t$, we get that 
$f^\ast\in \Loneloc(0,\mu(X))$ and 
$$
\sup_{I\subset(0,\mu(X))}\Omega(f^\ast, I)\le \sup_{I\subset(0,\mu(X))}\sup_k \Omega(f_k^\ast, I)\le c_*\|f\|_{\BMO}.
$$
This proves Eq.~\eqref{eq:BMO-bound} for $f$.
\end{proof}


\section{Boundedness of the decreasing rearrangement on BMO}


In this section, we establish the boundedness of
the decreasing rearrangement on $\BMO(X)$ 
under suitable assumptions on the basis $\cA$. 
A sufficient condition is 
provided by the following decomposition.

\begin{definition} 
\label{criterion}
Let $(X,\cM, \mu)$ be a measure space,
$\cA$ a basis in $X$, $f$ a 
nonnegative measurable function
on $X$, and $c_\ast\geq{1}$. We say that $\cA$ admits a 
{\em $c_\ast$-Calder\'{o}n-Zygmund decomposition} 
for $f$ at a level $\gamma>0$ if there exist a 
pairwise-disjoint sequence $\{A_i\}\subset\cA$ and a corresponding 
sequence $\{\widetilde{A}_i\}\subset\cA$ such that
\begin{itemize}
\item[(i)] 
for all $i$, $\widetilde{A}_i\supset A_i$ and $\mu(\widetilde{A}_i)
\leq c_*\mu(A_i)$,
\item[(ii)] for all $i$, 
$f$ is integrable on $\widetilde A_i$ and
$\displaystyle{\fint_{\widetilde{A}_i}\!f \le \gamma
\le \fint_{A_i}\!f}$,
\end{itemize}
and 
\begin{itemize}
\item[(iii)] $f\leq \gamma$ almost 
everywhere on $X\setminus\bigcup \widetilde{A}_i$.
\end{itemize}
\end{definition} 

By this nomenclature, when $X$ is a cube in $\R^n$ with Lebesgue measure, the classical Calder\'{o}n-Zygmund lemma states that the basis of cubes in $X$ admits a $2^n$-Calder\'{o}n-Zygmund decomposition for any nonnegative integrable $f$ and $\gamma>f_{X}$. The multidimensional Riesz rising sun lemma \cite{kls} states that when $X$ is a rectangle in $\R^n$ with Lebesgue measure, the basis of rectangles in $X$ admits a $1$-Calder\'{o}n-Zygmund decomposition for any integrable $f$ and $\gamma>f_{X}$.

\subsection{Basic oscillation estimate}

The heart of the proof of the boundedness 
criterion, Theorem~\ref{thm:BMO-bound} below,
is the following basic estimate. Its proof
relies on an argument developed by Klemes \cite{kl}
for the decreasing rearrangement of a nonnegative function in one dimension,
with Definition~\ref{criterion} in place of the rising sun lemma.

\begin{lemma}
\label{lem:osc} 
Let $g\in L^\infty(X)$ be nonnegative, $\cA$ a basis in $X$, 
$0<t<\mu(X)$, and $c_\ast\ge 1$.
If $\cA$ admits a $c_\ast$-Calder\'{o}n-Zygmund 
decomposition $\{A_i,\widetilde{A}_i\}$  at 
level $\gamma=(g^\ast)_{(0,t)}$, then
\begin{equation}
\label{eq:osc}
\Omega(g^*,(0,t))\le c_* \sup_{i} \Omega(g, \widetilde{A}_i)\,.
\end{equation}
\end{lemma}

\begin{proof} 
If $g^*$ is constant on $(0,t)$, there is nothing to show.
Otherwise, we separately consider the
numerator and denominator
in the definition 
of $\Omega(g^*,(0,t))$.

By assumption, there exist a sequence of
pairwise-disjoint subsets $\{A_i\}\subset\cA$
and another sequence $\{\widetilde{A}_i\}\subset\cA$
satisfying Conditions (i), (ii), and (iii). 
Set $E=\bigcup A_i$ and $\widetilde E=\bigcup \widetilde {A}_i$.

By Condition (iii), the set $\{x\in X: g(x)>\gamma\}$ is contained, up to a set of measure zero, in $\widetilde{E}$. From the equimeasurability of $g^*$ with $g$, it follows that 
$$
\int_{0}^{t}\!(g^\ast-\gamma)_+ 
=\int_{\gamma}^\infty \mu_g(\alpha)\, d\alpha
\le \int_{\widetilde{E}}\!(g-\gamma)_+ 
\le \sum_{i} \int_{\widetilde{A}_i}(g-\gamma)_+\,.
$$
Since, for each $i$, $\mu(\widetilde{A}_i)\le c_*\mu(A_i)$
by Condition (i) and $f_{\widetilde A_i}\le \gamma$ by Condition (ii), we have that
$$
\int_{0}^{t}\!(g^\ast-\gamma)_+ \le 
\sum_{i}\mu(\widetilde{A}_i) \fint_{\widetilde{A}_i}(g-g_{\widetilde{A}_i})_+
\le c_*\sum_{i} \mu(A_i) \fint_{\widetilde{A}_i}(g-g_{\widetilde{A}_i})_+
\,.
$$

Using that $\sum \mu(A_i)=\mu(E)$ and Eq.~\eqref{eq:osc-alt}, we conclude that
\begin{equation}\label{eq:key-numerator}
2\int_{0}^{t}\!(g^\ast-\gamma)_+ 
\le c_* \mu(E) \sup_i \Omega(g, \widetilde {A}_i)\,.
\end{equation}
It remains to show that $\mu(E)\le t$, which in combination with Eq.~\eqref{eq:key-numerator} and another application of Eq.~\eqref{eq:osc-alt}, yields the result of this lemma.

If we knew  that $\mu(E) < \infty$, we could write, by the Hardy-Littlewood inequality and Condition (ii),
$$
g^\ast_{(0,\mu(E))}\ge g_{E}\ge \gamma\,.
$$
Since $g^*_{(0,\tau)}<\gamma$ for all $\tau>t$, it would 
follow that $\mu(E)\le t$.  Otherwise, we just apply this argument
with $E$ replaced by the sets of finite measure
$E_n=\cup_{i=1}^{n}A_i$ for $n \in \N$, to get
$
\mu(E)=\lim\limits_{n\rightarrow\infty}\mu(E_n)\leq{t}$.
\end{proof}

We also need a result to replace Lemma~\ref{lem:osc} in situations
where Calder\'{o}n-Zygmund decompositions are only available locally.

\begin{lemma}
\label{lem:osc-approx}
Let $g\in L^\infty(X)$ be nonnegative, $\cA$ a basis in $X$, $0<t<\mu(X)$, $\gamma=(g^*)_{(0,t)}$, and $c_\ast\ge 1$.
If there is a sequence of nonnegative measurable functions with $g_k\uparrow g$ pointwise
such that $\cA$ admits a $c_*$-Calder\'{o}n-Zygmund
decomposition $\{A^k_i,\widetilde{A}^k_i\}$ for $g_k$
at level $\gamma$,  with the additional property that
\begin{itemize}
\item[(iv)] $g_k\equiv g$ on $\bigcup_i \widetilde A^k_i$,
\end{itemize}
then $\Omega(g^*, (0,t)) \le c_*\|g\|_{\BMO}$.
\end{lemma}

\begin{proof} 
From Lemma~\ref{lem:osc}, applied to each $g_k$ with decomposition $\{A^k_i,\widetilde{A}^k_i\}$, and the fact that $g\equiv g_k$ on each $\widetilde{A}^k_i$, we get
$$\Omega(g_k^*,(0,t))\le c_* \sup_{i} \Omega(g_k, \widetilde{A}^k_i)\le c_*\|g\|_{\BMO}.$$  
Since the $g_k$ are nonnegative and bounded, the hypotheses of Lemma~\ref{lem:mono} apply, so
$$\Omega(g^\ast,(0,t)) 
\le c_*\|g\|_{\BMO}.$$ 
\end{proof}

\subsection{General boundedness criterion}

Lemmas~\ref{lem:BMO-bound} and \ref{lem:osc} now combine to give us the following result.

\begin{theorem}
\label{thm:BMO-bound}
Let $X$ be a semi-finite measure space, $\cA$ a basis in $X$, 
and \mbox{$c_*\ge 1$}. Assume that 
for every nonnegative $g\in L^\infty(X)$
and each $0<t<\mu(X)$, the basis $\cA$ admits a 
$c_\ast$-Calder\'{o}n-Zygmund decomposition at level 
$\gamma=(g^*)_{(0,t)}$.

If $f\in\BMO(X)$ is rearrangeable, then $f^\ast$ 
is locally integrable and 
\begin{equation}
\label{eq:BMO-bound2}
\|f^\ast\|_{\BMO}\le c_*\|f\|_{\BMO}\,.
\end{equation}
\end{theorem}

\begin{remark}\label{relaxation}
By replacing Lemma~\ref{lem:osc} with Lemma~\ref{lem:osc-approx},
the conclusion of
Theorem~\ref{thm:BMO-bound} holds under the
weaker hypothesis that every nonnegative $g\in L^\infty(X)$ satisfies the assumption of Lemma~\ref{lem:osc-approx}.
\end{remark}

\section{Application to metric measure spaces}\label{sec:mms}

Let $(X,\rho)$ be a metric space equipped with 
a nontrivial Borel regular measure~$\mu$. A {\em closed ball} in $X$ is a subset of the form
$$
B(x,r)=\left\{y\in X: \rho(x,y)\leq r\right\}\,
$$
for some prescribed {\em radius} $r>0$ and {\em centre} $x\in X$. 
We also make the assumption that 
$$0<\mu(B(x,r))<\infty \quad \mbox{ for all }
x\in{X}, \; r>0\,.$$
It follows that the measure is 
$\sigma$-finite: write  $X = \cup_{k \ge 1}B(x_0,k)$ for some 
$x_0\in X$. Note that the metric space is not assumed to be complete; in particular, domains in Euclidean space are examples of metric measure spaces. 

The collection of all balls
$\cB=\{B(x,r)\subset X: x\in X, r>0\}$ forms a basis in $X$, and we define the space $\BMO(X)$ with respect to $\cB$ as in Definition~\ref{def:BMO}.

\subsection{Doubling spaces}

We say that a measure is {\em doubling} if the measure of any ball
controls, up to a multiplicative 
constant, the measure
of the co-centred ball of twice the radius. Equivalently, there are constants $c_\lambda\ge 1$ such that
$$
\mu(B(x,\lambda r))\le c_\lambda \mu(B(x,r))
$$
for all $x\in X$, $r>0$, and $\lambda\geq 1$. The growth of $c_\lambda$ as $\lambda\to\infty$
provides a rough bound on the dimension of the space. Note that a doubling metric measure space is an example of a space of homogeneous type in the sense of Coifman-Weiss~\cite[Chapitre III]{cw}. 

Over the last few decades, doubling measures have 
received considerable attention in geometric analysis, in connection with
Monge-Amp\`ere equations~\cite{CG}, with properties of harmonic measure
on the boundary of domains~\cite{KT}, and with the theory of Sobolev spaces~\cite{hk}. There are many results in the literature about the existence of doubling measures (see, for example, ~\cite{ls}). In general, however, if a doubling measure $\mu$ is restricted to a subset of $A\subset X$, it may no longer be doubling. 

In any metric space, the basic covering theorem~\cite[Theorem 1.2]{hei}
implies that for every family $\cF$ of 
balls in $X$ of uniformly bounded radii, 
there exists a pairwise-disjoint subfamily $\cG$ in
$\cF$ such that
\begin{equation}\label{eq:vitali}
\bigcup_{B(x,r)\in\cF}B(x,r)\ \subset\ \bigcup_{B(x,r)\in\cG}B(x, 5r).
\end{equation}
In fact, the constant $5$ can be replaced
by any $\lambda>3$~\cite{bb}. 

A doubling metric measure space 
is also geometrically doubling, in the sense that any ball can be covered
by a fixed finite number of balls of half the radius \cite{cw}, and so any disjoint collection of balls is necessarily countable \cite{hy}. Therefore, the subfamily $\cG$ coming from the basic covering theorem is countable when $\mu$ is doubling.

The Lebesgue differentiation theorem is well known to hold in the setting of doubling metric measure spaces \cite[Theorem 1.8]{hei}: 
for every locally integrable function $f$ on $X$,
$$
\lim_{r\to 0^+} f_{B(x,r)} = f(x)
$$
holds for almost every $x\in X$.

\subsection{Proof of the main result}

\begin{proof}[Proof of Theorem~\ref{thm:bounded-doubling}]
We verify that $X$ satisfies the relaxed assumptions 
of Theorem~\ref{thm:BMO-bound} mentioned in Remark~\ref{relaxation} with the basis $\cB$ of all balls and $c_*=c_5$.

Fix $0<t<\mu(X)$. Let $g\in L^\infty(X)$ be nonnegative and set $\gamma=(g^*)_{(0,t)}$. We construct a nonnegative monotone 
sequence $g_k\uparrow g$ and
a $c_*$-Calder\'{o}n-Zygmund decomposition
for each $g_k$ at level $\gamma$ satisfying (iv) of Remark~\ref{relaxation}. 

If $\mu(E_\gamma(g))=0$, there is nothing to show.
Otherwise, set
$$
r(x):= \inf \bigl\{r>0: g_{B(x,5r)}\le \gamma\bigr\}\,,\qquad x\in X\,.
$$
Since $\mu(B(x,r)) \rightarrow \mu(X)$ as $r \rightarrow \infty$, for $r$ sufficiently large we have $\mu(B(x,r)) \ge t$.  For such $r$, the monotonicity of $g^*$ and the Hardy-Littlewood inequality imply that
$$
\gamma \ge (g^*)_{(0,\mu(B(x,r)))} \ge g_{B(x,r)}\,,
$$
showing that $r(x)<\infty$.  Moreover, if $r(x) > 0$, 
\begin{equation}
\label{eq:5}
g_{B(x,5r(x))}\le \gamma < g_{B(x,r(x))}\,,
\end{equation}
where the first inequality holds since the map
$r\mapsto\mu(B(x,r))$ is right-continuous for any $x\in X$, and
the second holds by the definition of $r(x)$.  

By the Lebesgue differentiation theorem, $r(x)>0$ for almost
every $x\in E_\gamma(g)$,  so the collection
$$\cF:=\{B(x,r(x)): x\in E_{\gamma}(g),r(x) > 0\}$$ 
covers $E_\gamma(g)$ up to a set of measure zero, and $g \leq \gamma$ almost everywhere on 
$$S : = X \setminus \bigcup_{\cF} B(x,5r(x))\, .$$
As $X$ may have infinite diameter, there is no guarantee that
the radii of the balls in the collection $\cF$ are uniformly bounded.
For $k \in \N$, consider the subcollection
$\cF_k$ consisting of those balls in $\cF$ whose radii are bounded above by $k$,  and let
$$
X_k:= \bigcup_{\cF_k } B(x,5r(x)) \cup S\,,\qquad g_k = g\mathcal{X}_{X_k}\,,
$$
so that $g_k\uparrow g$.  By the basic covering theorem,
there exists a countable pairwise-disjoint subfamily $\{B(x_i, r(x_i))\}$ 
of $\cF_k$ satisfying Eq.~\eqref{eq:vitali}. Set $B_i:= B(x_i,r(x_i))$, $\widetilde B_i:=B(x_i,5r(x_i))$.  Then $g_k = g$ on $\cup \widetilde B_i$ and the level set
$E_\gamma(g_k)$ is contained, up to a set of measure zero, in $\cup \widetilde B_i$.  Combining this with Eq.~\eqref{eq:5}, we obtain a $c_*$-Calder\'{o}n-Zygmund decomposition for  $g_k$ at level $\gamma$, with $c_*=c_5$.
By Lemma~\ref{lem:osc-approx} and Remark~\ref{relaxation}, this completes the proof. 

Note that if we replace $5$ by any $\lambda > 3$ in the basic covering  theorem, we get Eq.~\eqref{eq:BMO-bounded} with $c_*= c_\lambda$, and taking the infimum over all such $\lambda$ gives the same conclusion with $c_*$ as in Eq.~\eqref{eq:cstar}.
\end{proof}

\bibliographystyle{amsplain}

\end{document}